\documentclass[reqno, 11pt]{amsart}
\makeatletter
\let\origsection=\section 
\def\section{\@ifstar{\origsection*}{\mysection}}
\def\mysection{\@startsection{section}{1}\z@{.7\linespacing\@plus\linespacing}{.5\linespacing}{\normalfont\scshape\centering\S}}
\makeatother

\usepackage{amsmath,amssymb,amsthm}
\usepackage{mathrsfs}
\usepackage{dsfont}
\usepackage{mathabx}\changenotsign

\usepackage{microtype}

\usepackage{xcolor}
\usepackage[backref=page]{hyperref}
\hypersetup{%
    colorlinks,
    linkcolor={red!50!black},
    citecolor={green!50!black},
    urlcolor={blue!50!black}
}

\usepackage{varioref}

\usepackage[english,algoruled,vlined,resetcount,linesnumbered]{
algorithm2e}

\usepackage{bookmark}
\usepackage{mathtools}
\usepackage[abbrev,msc-links,backrefs]{amsrefs}
\usepackage{doi}

\renewcommand{\PrintDOI}[1]{\doi{#1}}

\usepackage[T1]{fontenc}
\usepackage{lmodern}

\usepackage[english]{babel}

\usepackage{tikz}

\usepackage{fullpage}
\usepackage{setspace}

\makeatletter
\def\@setdate{\datename\ \@date}
\def\@setaddresses{\par
  \nobreak \begingroup
\footnotesize
  \def\author##1{\nobreak\addvspace\bigskipamount}%
  \def\\{\unskip, \ignorespaces}%
  \interlinepenalty\@M
  \def\address##1##2{\begingroup
    \par\addvspace\bigskipamount\indent
    \@ifnotempty{##1}{(\ignorespaces##1\unskip) }%
    {\scshape\ignorespaces##2}\par\endgroup}%
  \def\curraddr##1##2{\begingroup
    \@ifnotempty{##2}{\nobreak\indent{\itshape Current address}%
      \@ifnotempty{##1}{, \ignorespaces##1\unskip}\/:\space
      ##2\par}\endgroup}%
  \def\email##1##2{\begingroup
    \@ifnotempty{##2}{\nobreak\indent{\itshape Email addresses}%
      \@ifnotempty{##1}{, \ignorespaces##1\unskip}\/:\space
      \ttfamily##2\par}\endgroup}%
  \def\urladdr##1##2{\begingroup
    \@ifnotempty{##2}{\nobreak\indent{\itshape URL}%
      \@ifnotempty{##1}{, \ignorespaces##1\unskip}\/:\space
      \ttfamily##2\par}\endgroup}%
  \addresses
  \endgroup
}
\makeatother
\usepackage{enumitem}
\newcommand\rmlabel{\upshape({\itshape\roman*\,\/})}

\newcommand\alabel{\upshape({\itshape\alph*\,\/})}

\let\polishlcross=\l
\def\l{\ifmmode\ell\else\polishlcross\fi}

\renewcommand{\emptyset}{\varnothing}
\renewcommand{\setminus}{\smallsetminus}

\makeatletter
\def\moverlay{\mathpalette\mov@rlay}
\def\mov@rlay#1#2{\leavevmode\vtop{%
   \baselineskip\z@skip \lineskiplimit-\maxdimen
   \ialign{\hfil$\m@th#1##$\hfil\cr#2\crcr}}}
\newcommand{\charfusion}[3][\mathord]{
    #1{\ifx#1\mathop\vphantom{#2}\fi
        \mathpalette\mov@rlay{#2\cr#3}
      }
    \ifx#1\mathop\expandafter\displaylimits\fi}
\makeatother

\newcommand{\eps}{\varepsilon}
\let\epsilon\varepsilon
\let\tilde\widetilde
\let\ldots\dots

\newtheoremstyle{case}{}{}{\normalfont}{}{\itshape}{:}{ }{}

\newtheorem{thm}{Theorem}
\newtheorem{lem}[thm]{Lemma}
\newtheorem{prop}[thm]{Proposition}

\newtheorem{cor}[thm]{Corollary}

\newtheorem*{claim*}{Claim}

\theoremstyle{definition}

\newtheorem{ex}[thm]{Example}

\newtheoremstyle{case}{}{}{\normalfont}{}{\itshape}{\normalfont:}{ }{}

\theoremstyle{case}

\newcommand{\oldqed}{}
\def\endofClaim{\hfill\scalebox{.6}{$\Box$}}

\newenvironment{claimproof}[1][Proof]{
  \renewcommand{\oldqed}{\qedsymbol}
  \renewcommand{\qedsymbol}{\endofClaim}
  \begin{proof}[#1]
}{
  \end{proof}
  \renewcommand{\qedsymbol}{\oldqed}
}

\let\:\colon
\let\subset\subseteq

\def\({\left(}
\def\){\right)}
\let\<\langle
\let\>\rangle

\newcommand{\hatp}{\hat{p}}

\newcommand{\NN}{\mathbb{N}}
\newcommand{\PP}{\mathbb{P}}

\def\cF{\mathcal{F}}
\def\cG{\mathcal{G}}
\def\cT{\mathcal{T}}
\def\dist{\text{\rm dist}}
\usepackage{datetime}
\usepackage{lineno}
\newcommand*\patchAmsMathEnvironmentForLineno[1]{%
\expandafter\let\csname old#1\expandafter\endcsname\csname #1\endcsname
\expandafter\let\csname oldend#1\expandafter\endcsname\csname end#1\endcsname
\renewenvironment{#1}%
{\linenomath\csname old#1\endcsname}%
{\csname oldend#1\endcsname\endlinenomath}}%
\newcommand*\patchBothAmsMathEnvironmentsForLineno[1]{%
\patchAmsMathEnvironmentForLineno{#1}%
\patchAmsMathEnvironmentForLineno{#1*}}%
\AtBeginDocument{%
\patchBothAmsMathEnvironmentsForLineno{equation}%
\patchBothAmsMathEnvironmentsForLineno{align}%
\patchBothAmsMathEnvironmentsForLineno{flalign}%
\patchBothAmsMathEnvironmentsForLineno{alignat}%
\patchBothAmsMathEnvironmentsForLineno{gather}%
\patchBothAmsMathEnvironmentsForLineno{multline}%
}

\begin{document}
\onehalfspace

\title{Universality for bounded degree spanning trees in randomly perturbed graphs}

\author[J. B\"ottcher]{Julia B\"ottcher}
\author[J. Han]{Jie Han}
  \author[Y. Kohayakawa]{Yoshiharu Kohayakawa}
\author[R. Montgomery]{Richard Montgomery}
\author[O. Parczyk]{Olaf Parczyk}
\author[Y. Person]{Yury Person}

\thanks{%
JB is partially supported by EPSRC (EP/R00532X/1).
JH is supported by FAPESP (2014/18641-5, 2013/03447-6).
YK is partially supported by FAPESP (2013/03447-6) and
  CNPq (310974/2013-5, 311412/2018-1, 423833/2018-9).
OP and YP are supported by DFG grant PE 2299/1-1.
The cooperation of the authors was supported by a joint CAPES-DAAD
  PROBRAL project (Proj.\ no.~430/15, 57350402).
FAPESP is the S\~ao Paulo Research Foundation.   
CNPq is the National Council for Scientific and Technological
Development of Brazil.%
}

\shortdate
\yyyymmdddate
\settimeformat{ampmtime}
\date{\today, \currenttime}
\footskip=28pt

\address{Department of Mathematics, London School of Economics, Houghton Street,
London WC2A 2AE, UK}
\email{j.boettcher@lse.ac.uk}

\address{Department of Mathematics, University of Rhode Island, 5
  Lippitt Road, Kingston, RI, USA, 02881} 
\email{jie\_han@uri.edu}

\address{Instituto de Matem\'atica e Estat\'{\i}stica, Universidade de
  S\~ao Paulo, Rua do Mat\~ao 1010, 05508-090 S\~ao Paulo, Brazil}
\email{yoshi@ime.usp.br}

\address{University of Birmingham, Birmingham, B15 2TT, UK}
\email{r.h.montgomery@bham.ac.uk}

\address{Institut f\"ur Mathematik, Technische Universit\"at Ilmenau, 98684 Ilmenau, Germany}
\email{\{olaf.parczyk|yury.person\}@tu-ilmenau.de}

\begin{abstract}
  We solve a problem of Krivelevich, Kwan and
  Sudakov [SIAM Journal on Discrete Mathematics 31 (2017), 155--171]
  concerning the threshold for 
  the containment of all bounded degree spanning trees in the model of
  randomly perturbed dense graphs. More precisely, we show that, if we
  start with a dense graph $G_\alpha$ on $n$ vertices with
  $\delta(G_\alpha)\ge \alpha n$ for $\alpha>0$ and we add to it the
  binomial random graph $G(n,C/n)$, then with high probability the graph
  $G_\alpha\cup G(n,C/n)$ contains copies of all spanning trees with
  maximum degree at most $\Delta$ simultaneously, where $C$ depends
  only on $\alpha$ and $\Delta$.
\end{abstract}

\maketitle

\section{Introduction}
Many problems from extremal graph theory concern Dirac-type
questions. These ask for asymptotically optimal conditions on the
minimum degree $\delta(G_n)$ for an $n$-vertex graph $G_n$ to contain
a given spanning graph $F_n$.  Typically, there exists a constant
$\alpha>0$ (depending on the family $(F_i)_{i \ge 1}$) such that
$\delta(G_n)\ge \alpha n$ 
implies $F_n\subseteq G_n$.  A prime example is Dirac's
theorem~\cite{dirac1952some} stating that $\delta(G_n)\ge n/2$ ensures
that $G_n$ is Hamiltonian if~$n\geq3$.

On the other hand, a large branch of the theory of random graphs
studies when random graphs typically contain a copy of a given
spanning structure $F_n$.  Let~$G(n,p)$ be the $n$-vertex binomial
random graph, where each of the $\binom{n}{2}$ possible edges is
present independently at random with probability $p=p(n)$.  A
classical result of Bollob\'as and
Thomason~\cite{bollobas1987threshold} states that every nontrivial
monotone property has a threshold in $G(n,p)$. Since containing a copy
of (a sequence of graphs) $F_n$ is a monotone property, there exists a
\emph{threshold function} $\hatp=\hatp(n)\colon \NN\to[0,1]$ such
that, if $p=o(\hatp)$, then $\lim_{n\to\infty}\PP[F_n\subseteq G(n,p)]=0$, whereas,
if $p=\omega(\hatp)$, then
$\lim_{n\to\infty}\PP[F_n\subseteq G(n,p)]=1$.  When the
conclusion of the latter case holds, we say that $G(n,p)$
contains~$F_n$ \emph{asymptotically almost surely} (a.a.s.). For
example, a famous result of Kor\v{s}unov~\cite{Kor76} and
P\'{o}sa~\cite{Pos76} asserts that the threshold for Hamiltonicity in
$G(n,p)$ is $(\log n)/n$.

Bohman, Frieze and Martin discovered the following phenomenon
in~\cite{bohman2003many}.  Given a fixed $\alpha>0$, they started with
a graph $G_\alpha$ on $n$ vertices with
$\delta(G_\alpha)\ge \alpha n$.  Here, $\alpha$ can be arbitrarily
small and hence~$G_\alpha$ can be far from containing any Hamilton
cycle.  They proved that, after adding $m=C(\alpha)n$ edges uniformly
at random to $G_\alpha$, the new graph $G$ becomes Hamiltonian a.a.s.,
where~$C(\alpha)$ is a constant that depends only on~$\alpha$.
Letting~$G_\alpha$ be the complete unbalanced bipartite
graph~$K_{\alpha n,\,(1-\alpha)n}$, one sees that the addition of
linearly many edges to~$G_\alpha$ is necessary for this result to hold
in general.  Furthermore, clearly, the conditions on
$\delta(G_\alpha)$ and on $p=m/\binom{n}{2}$ in this result are weaker
than in the corresponding Dirac-type problem and the threshold
problem, respectively.  More precisely, the probability $p$ turns out
to be smaller by a factor of $\Theta(\log n)$.  Here, we have switched
from choosing $m$ edges uniformly at random to the the binomial
$G(n,p)$ model, which is known to be essentially equivalent when
$p=m/\binom{n}{2}$ (see, e.g.,~\cite{janson2011random}).

The model $G_\alpha\cup G(n,p)$ is known as the \emph{randomly
  perturbed graph model}. Typically $p=o(1)$, so an `addition' of
$G(n,p)$ to the dense graph $G_\alpha$ corresponds to a small random
perturbation in the structure of $G_\alpha$.  This model and its
related generalizations to hypergraphs and digraphs sparked a great
deal of research in recent years.

In this paper we are concerned with spanning trees in randomly
perturbed graphs.  For \textit{almost spanning} trees it was shown by
Alon, Krivelevich and Sudakov~\cite{AKS07} that, for some constant
$C=C(\eps,\Delta)$, the random graph $G(n,C/n)$ alone a.a.s.\ contains
any tree with at most $(1-\eps)n$ vertices and maximum degree at most
$\Delta$, where the bounds on $C=C(\eps,\Delta)$ have subsequently
been improved~\cite{BCPS10}. Since the random graph $G(n,C/n)$ a.a.s.\
contains isolated vertices, it obviously does not contain spanning
trees. The problem of determining the threshold of bounded degree
spanning trees attracted much attention. Recently,
Montgomery~\cite{MontgomeryTrees} showed that for each constant
$\Delta$ and every sequence of trees~$T_n$ with maximum
degree~$\Delta$, the threshold in $G(n,p)$ for a copy of~$T_n$ to
appear is $(\log n)/n$ (see also~\cite{M14a}).  However, Krivelevich,
Kwan and Sudakov~\cite{krivelevich2015bounded} showed that, again, a
smaller probability suffices in the randomly perturbed graph model.
They proved that $G_\alpha\cup G(n,p)$ a.a.s.\ contains a given
spanning tree~$T_n$ with maximum degree at most~$\Delta$ when
$p=C(\Delta,\alpha)/n$.

In the concluding remarks of~\cite{krivelevich2015bounded},
Krivelevich, Kwan and Sudakov raised the question of whether
$G_\alpha\cup G(n,D/n)$ contains all spanning trees of maximum degree
at most $\Delta$ simultaneously, for some constant
$D=D(\Delta,\alpha)$. The purpose of this paper is to answer their
question in the affirmative.  For stating our result we need some
notation. For a family $\cF$ of graphs, we say that a graph~$G$ is
\emph{$\cF$-universal} if $G$ contains a copy of every graph $F$ from
$\cF$. We denote by $\cT(n,\Delta)$ the family of all trees of maximum
degree at most $\Delta$ on $n$ vertices.

\begin{thm}\label{thm:main_theorem}
  For each $\alpha>0$ and $\Delta\in \NN$, there exists a constant
  $D=D(\Delta,\alpha)$ such that the following holds.  If~$G_\alpha$
  is an $n$-vertex graph with $\delta(G_\alpha)\ge \alpha n$, then the
  randomly perturbed graph $G_\alpha\cup G(n,D/n)$ is a.a.s.\
  $\cT(n,\Delta)$-universal.
\end{thm}

This result is asymptotically optimal for $0<\alpha<1/2$, as with
$G_\alpha$ the complete unbalanced bipartite graph $K_{\alpha
  n,(1-\alpha) n}$ we need a linear number of edges from $G(n,p)$
already for the perfect matching.
For $\alpha > 1/2$ this follows from a more general result, applied to 
$G_\alpha$ alone, due to Koml\'os, S\'ark\"ozy, and
Szemer\'edi~\cite{KSStrees}, for trees with maximum degree up to $n /
\log n$. 
Kim and Joos~\cite{JK_Trees} have succeeded in transferring this
result to the perturbed model. 

Theorem~\ref{thm:main_theorem} is an immediate consequence of a
technical theorem, 
Theorem~\ref{thm1}, which states that the union of $G_\alpha$ with any
reasonably expanding graph $G$ is $\cT(n,\Delta)$-universal. 
The proof of Theorem~\ref{thm1} relies on the use of reservoir sets
resembling those introduced in~\cite{BMPP17} as part of the so-called
assisted absorption method. The novelty in our proof is that we construct
these reservoir sets using expanding graphs rather than random graphs,
which is not possible with the techniques from~\cite{BMPP17} (see also the
discussion in Section~\ref{sec:overview} and the proof of Lemma~\ref{lem3} in Section~\ref{sec:reservoir}).

Before we turn to the details of our embedding technique, we mention
further results concerning randomly perturbed graphs. Further spanning
structures whose appearance in randomly perturbed graphs has been
studied are $F$-factors (for fixed graphs $F$)~\cite{BTW_tilings},
squares of Hamilton cycles and copies of general bounded degree
spanning graphs~\cite{BMPP17}, perfect matchings and loose Hamilton
cycles in uniform hypergraphs~\cite{krivelevich2015cycles}, and tight
Hamilton cycles in hypergraphs~\cite{HanZhao2017}. Most of the
mentioned results exhibit the following phenomenon: in the presence of
a dense graph~$G_\alpha$, a smaller edge probability than in~$G(n,p)$
alone suffices.  The only exception to this rule so far are
$F$-factors for certain non-strictly-balanced graphs~$F$ covered
in~\cite{BTW_tilings}.  Moreover, some variations of such results when
$\alpha$ is at least some positive constant $c$ (which depends on
other parameters of the problems at hand) were considered
in~\cites{BHKM_powers,MM_HyperPerturbed}.


\section{Notation, main technical result, and proof overview}
\label{sec:overview}
We will use standard graph theoretic notation throughout. In the
following, we briefly recap most of the relevant terminology. Given
graphs $G$ and $H$, write $|G| = |V(G)|$ and
$G\setminus H = G[V(G)\setminus V(H)]$, that is, the induced subgraph
of $G$ on $V(G)\setminus V(H)$.  Throughout this note we omit floors
and ceilings.  For two not necessarily disjoint sets $U$ and $W$ of
vertices of a graph $G$ we write $e(U,W)$ for the number of
edges with one endpoint in $U$ and the other in $W$, where we count
edges that lie in $U\cap W$ twice.

We say that an $n$-vertex graph $G$ is an
\emph{$(n, p, \eps, C)$-graph} if $\Delta(G)\le C p n$ and, for any
$U,\,W\subset V(G)$ such that $|U|,\,|W|\ge \eps n$, we have
$e(U, W)\ge (p/C) |U| |W|$. We further denote the family of
$(n, p, \eps, C)$-graphs by $\cG(n, p, \eps, C)$.  Intuitively, the
graphs from $\cG(n, p, \eps, C)$ are graphs with a certain degree
bound which are expanding for vertex subsets of linear size.

Our main technical result states that perturbing graphs $G_\alpha$ with minimum degree
at least $\alpha n$ by graphs $G\in \cG(n, D/n, \eps, C)$ results in
$\cT(n,\Delta)$-universal graphs.

\begin{thm}[Main technical result]
  \label{thm1}
  For any $\alpha>0$ and integers $C\ge 2$ and~$\Delta\ge 1$, there
  exist $\eps>0$, $D_0$ and~$n_0$ such that the following holds for
  any $D\ge D_0$ and $n\ge n_0$.  Suppose 
  $G\in \cG(n, D/n, \eps, C)$ and $G_\alpha$ are $n$-vertex graphs
  on the same vertex set and $\delta(G_\alpha) \ge \alpha n$.
  Then $H:=G_\alpha\cup G$ is $\cT(n, \Delta)$-universal.
\end{thm}

We will show in Section~\ref{sec:expanderToRandom} that this
result implies Theorem~\ref{thm:main_theorem}.
In the remainder of this section, we give a brief outline of our
proof of Theorem~\ref{thm1}. 

\subsection{Proof overview}
\label{sec:pf_overview}
Let $G\in\cG(n,D/n,\eps,C)$.
We embed an arbitrary $T\in\cT(n,\Delta)$  into $H:=G_\alpha\cup G$ in
three phases. In the first phase, we find a subtree $T_1$ of $T$ (see
Lemma~\ref{lem1}) of small linear size, say $\beta n$ with $\beta\ll
\alpha$, and we embed this subtree $T_1$ into $H$ using a randomized
algorithm (see Lemma~\ref{lem3}).
In doing so, we can show that there is some such embedding in which,
for any given pair of vertices $u$, $v\in V(H)$, there are at least
$3\Delta \eps n$ vertices $w\in V(T_1)$ with $N_T(w)\subset V(T_1)$
such that $w$ is embedded into 
$N_{H}(u)$ and $N_T(w)$ is embedded into $N_{H}(v)$ -- a fact which will
turn out to be crucial later.
We denote by $B_{T,H}(u,v)$ such a set of
vertices $w$, and refer to such sets $B_{T,H}(u,v)$ as \emph{reservoir
  sets} (see Section~\ref{sec:reservoir} for the formal definition).
Alternatively, calling them \emph{switching sets} would emphasize that
each of them can only be used once. 

In the second phase, we extend the tree $T_1$ to an almost spanning
subtree $T'$ of $T$ with $|T\setminus T'|=2\eps n$.  For this purpose
we use a theorem of Haxell~\cite{haxell2001tree} (see
Corollary~\ref{cor} below), which ensures such almost spanning
embeddings exist given sufficient expansion in the host graph~$H$. 

Finally, in the third phase, we complete our embedding using a greedy
approach and the reservoir sets $B_{T,H}(u,v)$ for the following
\emph{swapping} trick: since $T'$ is a subtree of $T$, we can extend
it by consecutively appending degree-$1$ vertices and thus growing the
tree $T'$ into $T$.  Suppose
$T'=T_0'\subset\dots\subset T_{2\epsilon n}'=T$ is the sequence of
subtrees of~$T$ that we encounter in this process.  Suppose we already
have the embedding $g_{i-1}\: V(T_{i-1}')\to V(H)$, and we wish to
extend it to $g_i\: V(T_i')\to V(H)$ by defining the image of the
leaf~$b\in V(T_i')\setminus V(T_{i-1}')$.  
Given some vertex~$v$ of~$H$ available for
embedding~$b$ (that is, $v\notin g_{i-1}(V(T_{i-1}'))$), if
there is an edge in~$H$
from~$v$ to~$g_{i-1}(u)$, where~$u$ is the parent of~$b$ in~$T_i$,
then we simply embed~$b$ onto~$v$ (that is, we let~$g_i(b)=v$).
On the other hand, if there is no edge in~$H$ from~$v$
to~$g_{i-1}(u)$, we proceed as follows.  We will set things up so
that, by counting, we will 
be able to show that
there is some~$c\in V(T_{i-1})$ such that~$c \in
B_{T,H}(g_{i-1}(u),v)$.  We then let~$g_i(b)=g_{i-1}(c)$ and we
let~$g_i(c)=v$.  This defines a valid embedding $g_i\colon V(T_i)\to
V(H)$.  (We remark that we said that we would
\textit{extend}~$g_{i-1}$ to~$g_i$; as it will be clear by now, this
is not strictly speaking correct, as we may alter~$g_{i-1}$ slightly
before extending it to~$g_i$.)

\smallskip

As mentioned earlier, the reservoir sets used in our proof are similar to
those introduced in the setting of randomly perturbed graphs
in~\cite{BMPP17}. In that work, the reservoir sets are used to prove a
general result about spanning structures in randomly perturbed graphs,
which can be easily applied to consider the appearance of various different
single spanning structures. In particular, this gives a short proof of the
appearance of any single bounded degree spanning tree in this model, a
problem that was first solved in~\cite{krivelevich2015bounded}. The argument from~\cite{BMPP17} does not work for universality statements. However, here we show that the reservoirs can be found and the swapping trick employed in the completely deterministic setting by embedding the first part of the tree in a randomized way.

\section{Auxiliary Lemmas}

The lemmas provided in this section will be used in the proof of Theorem~\ref{thm1}.
We start in Section~\ref{sec:partition} with two lemmas for partitioning the tree~$T$ we want to embed. 
We then explain how we obtain good reservoir sets by embedding a subtree~$T_1$
of~$T$ randomly in Section~\ref{sec:reservoir}. Finally, in
Section~\ref{sec:almost} we provide the tools to extend this embedding to
an almost spanning subgraph of~$T$.

\subsection{Tree partitioning lemmas}\label{sec:partition}
 Recall that $\cT(n,\Delta)$ is the collection of all trees on $n$ vertices with maximum degree at most $\Delta$, and that a graph $G$ on $n$ vertices is said to be $\cT(n, \Delta)$-universal if $G$ contains a copy of $T$ for every $T\in \cT(n,\Delta)$.

The main assertion of the following lemma is that we can find in any
bounded degree tree $T$ a subtree $T_1$  of roughly any desired size so
that removing $T_1$ from~$T$ leaves a tree. We will use this lemma to find
a small linear sized subtree~$T_1$, which we embed in our first phase.

\begin{lem}\label{lem1}
  Let $\beta,\,\eps>0$ and let $n,\,\Delta$ be positive integers  such that $\Delta\beta + 2\eps <1$.
  Then, for any $T\in \cT(n, \Delta)$, there exist subtrees $T_1\subseteq T' \subseteq T$ such that
  \begin{enumerate}[label=\alabel]
  \item $\beta n\le |T_1|\le \Delta\beta n$,
  \item\label{lem1:b} $e(T_1, T\setminus T_1)=1$, and
  \item $|T\setminus T'|=2\eps n$.
  \end{enumerate}
\end{lem}

\begin{proof}
Fix any vertex $v$ of $T$ as the root and, for each $w\in V(T)$, write $C_w$ for the branch (subtree) of $T$ consisting of $w$ and all of its descendants.
By $\beta < 1/\Delta$, $\deg(v)\le \Delta$ and averaging, there is $v_1\in N_T(v)$ such that $|C_{v_1}|\ge \beta n$.
If $|C_{v_1}|> \Delta\beta n$, then similarly there is a child $v_2$ of $v_1$ such that $|C_{v_2}|\ge (\Delta\beta n-1)/(\Delta-1) \ge \beta n$.
Repeating this argument gives a desired $v'$ such that $\beta n\le |C_{v'}|\le \Delta\beta n$. Let $T_1:=C_{v'}$.
Note that~\ref{lem1:b} holds by the definition of $T_1$.
Finally, let $T'$ be an (arbitrary) subtree of $T$ such that $T_1\subseteq T' \subseteq T$ and $|T\setminus T'|=2\eps n$.
\end{proof}

Let $T$ be a tree.
Given vertices $x_1,\ldots,x_m$ of $T$, let $\langle x_1,\dots, x_m\rangle_T$ be the minimal subtree of $T$ that contains the vertices $x_1,\ldots,x_m$, which is just the subtree of $T$ obtained from the union of the vertex sets of all the paths between $x_i, x_j$, $i\neq j$, in $T$. For two distinct vertices $x, y$ of $T$, we write  $\dist_T(x, y)$ for  their distance in $T$, namely, the length of the (unique) path on $T$ connecting $x$ and $y$.
Given a vertex $x$ of $T$ and a vertex set $Y\subseteq V(T)$ such that $x\notin Y$, let $\dist_T(x, Y):=\min_{y\in Y} \dist_T(x, y)$.

The following lemma provides us with vertices $x_1,\dots,x_s$ in a tree~$T$
which cover~$T$ well, but are not too close. In particular, this gives us a
collection of stars $x_i\cup N_T(x_i)$ which are far enough apart that they
are relatively independent.

\begin{lem}\label{lem2}
  For any tree $T$ with maximum degree at most $\Delta$, there exist $s\in \mathbb{N}$ and vertices $x_1,\dots, x_s\in V(T)$ such that
  \begin{enumerate}[label=\alabel]
  \item for any $2\le i\le s$, $\dist_T(x_i, \langle x_1,\dots, x_{i-1}\rangle_T) = 5$,
  \item\label{lem2:b} $|T|/(5\Delta^4) \le s \le (|T|+4)/5$, and
  \item $\dist_T(x, \langle x_1,\dots, x_{s}\rangle_T) \le 4$ for all vertices $x\in V(T)$.
  \end{enumerate}
\end{lem}

\begin{proof}
We start with picking $x_1$ arbitrarily.
We greedily pick the vertices $x_2,\dots, x_s$ in $V(T)$ sequentially as long as there is a vertex $x_i$ such that $\dist_T(x_i, \langle x_1,\dots, x_{i-1}\rangle_T) = 5$.
Note that for any $2\le i\le s$, $|\langle x_1,\dots, x_{i}\rangle_T\setminus \langle x_1,\dots, x_{i-1}\rangle_T|=5$, so we inductively get that $|\langle x_1,\dots, x_{s}\rangle_T| = 5s-4$.
This implies $s\le (|T|+4)/5$.
Since $T$ is connected, the maximality of $s$ implies that $\dist_T(x, \langle x_1,\dots, x_{s}\rangle_T) \le 4$ for all vertices $x\in V(T)$.
Thus we have $|T|\le (5s-4) \Delta^4$, which implies~\ref{lem2:b}.
\end{proof}

\subsection{A randomized embedding -- controlling reservoir sets}\label{sec:reservoir}
In the following, we define formally the reservoir
sets~$B_{T,H}(u,v)$, already mentioned in the proof overview given in
Section~\ref{sec:pf_overview}, and show that we can force them to be
suitably large.  
These reservoir sets will be helpful when finishing the embedding
of~$T$, since they will allow us to alter locally partial embeddings that we
construct sequentially.  We warn the reader that, for technical convenience, the
sets~$B_{T,H}(u,v)$ are defined here in a slightly different manner
in comparison with the informal definition given earlier in 
Section~\ref{sec:pf_overview}.  
Let $V$ be a set of $n$ vertices.
Let $G$ be a graph on $V$ and let $T$ be a tree with $V(T)\subset V$.
For $v\in V$, let \[B_{T,G}(v):=\big\{w\in V(T): N_T(w)\subseteq N_{G}(v)\big\}\,.\]
For distinct vertices~$u$ and~$v\in V$, we define their
\emph{reservoir set}~$B_{T,G}(u, v)$ as follows:
\[B_{T,G}(u, v):= B_{T,G}(v)\cap N_{G}(u)\,.\]
Recall that the idea is that we can free up any $w\in B_{T,G}(u,v)$ used
already in the embedding, by moving the vertex embedded to~$w$ to~$v$. This
then allows us to use~$w$ for embedding any unembedded neighbour of the vertex
embedded to~$u$.

Our next lemma shows that we can embed the linear sized subtree~$T_1$
of~$T$ into $H=G\cup G_\alpha$ using a randomized algorithm, such that we
get large reservoir sets.

\begin{lem}\label{lem3}
  For any $\alpha>0$ and integers $C\ge 2$ and~$\Delta\ge 1$, there
  exist $\eps>0$, $D_0$ and~$n_0$, such that the following
  holds for $D\ge D_0$ and $n\ge n_0$.  Suppose
  $G\in \cG(n, D/n, \eps, C)$ and $G_\alpha$ is an $n$-vertex graph
  such that $\delta(G_\alpha) \ge \alpha n$ and $V(G)=V(G_\alpha)=:V$.
  Then, for any tree $T_1$ such that $\Delta(T_1)\le \Delta$ and
  $\alpha n/(2\Delta^2)\le |T_1|\le \alpha n/(2\Delta)$, there is an
  embedding $g$ of $T_1$ into $H:=G\cup G_\alpha$ such that
  $|B_{\tilde{T_1},H}(u,v)|\ge 2(\Delta+3)\eps n$ for any $u$ and~$v\in V$,
  where $\tilde{T_1}=g(T_1)$.
\end{lem}

\begin{proof}[Proof of Lemma~\ref{lem3}]
First we choose the parameters $D_0$ and $\eps$ as follows:
\begin{equation}\label{eq:eps_Dz}
D_0 := 2C\Delta/\alpha\quad\text{and}\quad \eps:=\alpha^{\Delta+2}C^{-2\Delta}2^{-\Delta-8}\Delta^{-7},
\end{equation}
and then we choose $n_0$ large enough.

We apply Lemma~\ref{lem2} to $T_1$ and obtain $s\in \mathbb{N}$ and vertices $x_1,\dots, x_s\in V(T_1)$ such that, for any $2\le i\le s$, $\dist_{T_1}(x_i, \langle x_1,\dots, x_{i-1}\rangle_{T_1}) = 5$, and $|T_1|/(5\Delta^4) \le s \le (|T_1|+4)/5$.
Our embedding of $T_1$ consists of three steps.
First we iteratively embed the disjoint stars with centers at $x_1,\dots, x_s$ uniformly at random into stars in $H$ (using only the edges of $G$) whose vertices have not yet been used as images.
Next we connect these stars and obtain an embedding of a subtree of $T_1$ as the union of the stars and $\langle x_1,\dots, x_{s}\rangle_{T_1}$.
At last we embed the rest of the vertices of $T_1$ greedily, which will be possible using $G_\alpha$ as $|T_1|\le \alpha n/(2\Delta)$ and $\delta(G_\alpha)\ge \alpha n$.

The following claim states that we can pick disjoint stars with $\Delta$
leaves (that is, copies of $K_{1,\Delta}$) in~$G$, within which we will later embed the stars in $T_1$ with centers at $x_1,\dots, x_s$.

\begin{claim*}
There is a choice of disjoint stars $S_1,\dots,
S_s$ with $\Delta$ leaves in~$G$ such that, for each $u, v\in V$ there are at least
$2(\Delta+3)\eps n$ stars among $S_1,\dots, S_s$ with their
centers  in $N_{G_\alpha}(u)$ and their leaves in $N_{G_{\alpha}}(v)$.
\end{claim*}
\begin{claimproof}[Proof of the Claim]
We randomly and sequentially pick $s$ stars $S_1,\dots, S_s$ with $\Delta$
leaves from $G$, where each star $S_i$ is picked uniformly at random from
the copies of $K_{1,\Delta}$ which are disjoint from $S_1,\ldots,S_{i-1}$
(we show below that this is indeed possible).

For $u,v\in V$, $i\in [s]$, let $Y_{i}^{u,v}$ be the Bernoulli random variable for the event that $\tilde{x}_i\in N_{G_\alpha}(u)$ and $R_i\subseteq N_{G_{\alpha}}(v)$, where $\tilde{x}_i$ is the center of $S_i$ and $R_i$ is the set of leaves of $S_i$.
Since $\delta(G_\alpha) \ge \alpha n$, $|T_1|\le \alpha n/(2\Delta)$ and the existing stars cover at most
\begin{equation*}
(\Delta+1)s \le (\Delta+1)\left(\frac{|T_1|}5+4\right)\le (\Delta+1)\left(\frac{{\alpha n}}{10\Delta}+4\right) \le \alpha n/4
\end{equation*}
vertices, there are at least $3\alpha n/4$ vertices available in both $U:=N_{G_\alpha}(u)\setminus \bigcup_{j\in [i-1]} V(S_j)$ and $W:=N_{G_\alpha}(v)\setminus \bigcup_{j\in [i-1]} V(S_j)$.

Since $G\in \cG(n, D/n, \eps, C)$ and $3\alpha/4\ge \eps$, $e(U, W)\ge D|U||W|/(Cn)\geq 3\alpha D|U|/(4C)$. By the convexity of the binomial function, the number of $K_{1,\Delta}$-stars with center in $U$ and leaves in $W$ is at least
\begin{equation*}
\sum_{u\in U}\binom{\deg_W(u)}{\Delta}\ge |U|\binom{\sum_{u\in U}\deg_W(u)/|U|}{\Delta}\geq |U|\binom{3\alpha D/(4C)}{\Delta}.
\end{equation*}
Since $\Delta(G)\le CD$, the total number of $K_{1,\Delta}$-stars in $G$ is at most $n\binom{CD}{\Delta}$. This allows us to obtain the following lower bound on $\mathbb{E} (Y_i^{u,v}\mid Y_1^{u,v},\dots, Y_{i-1}^{u,v})$:
\[
\mathbb{E} (Y_i^{u,v}\mid Y_1^{u,v},\dots, Y_{i-1}^{u,v}) \ge \frac{|U|}{n}\frac{\binom{3\alpha D/(4C)}{\Delta}}{\binom{CD}{\Delta}}\ge
2^{-\Delta-1}\alpha^{\Delta+1}C^{-2\Delta}.
\]

Let $p := 2^{-\Delta-1}\alpha^{\Delta+1}C^{-2\Delta}$ and
\begin{equation*}
x:=s p \ge \frac{|T_1|}{5\Delta^4} p \ge \frac{\alpha n}{2\Delta^2\cdot 5\Delta^4} \cdot \frac{\alpha^{\Delta+1}}{2^{\Delta+1}C^{2\Delta}}  \ge 4(\Delta+3)\eps n,
\end{equation*}
by the choice of $\eps$ in \eqref{eq:eps_Dz}.
Thus, by Lemma~2.2 (the sequential dependence lemma)
from~\cite{allen2016blow} with $\delta=1/2$, or a simple coupling
argument, we get 
\begin{equation*}
\mathbb{P}\big(Y_1^{u,v}+ \cdots + Y_{s}^{u,v} < 2(\Delta+3)\eps n \big) \le \mathbb{P}\big(Y_1^{u,v}+ \cdots + Y_{s}^{u,v} < x/2 \big) < e^{- x/12} \le e^{-\eps n}\,.
\end{equation*}
Thus by the union bound, we conclude that there is a choice of $S_1,\dots,
S_s$ such that, for each $u, v\in V$, $Y_1^{u,v}+ \cdots + Y_{s}^{u,v} \ge
2(\Delta+3)\eps n$, i.e., the claim holds.
\end{claimproof}

Now let $S_1,\dots,S_s$ be as given by the claim.
Define the embedding $g$ of the stars in~$T_1$ on vertices
$\{x_1\}\cup N_{T_1}(x_1)\cup\dots\cup \{x_s\}\cup N_{T_1}(x_s)$ by mapping
the star (which does not necessarily have  $\Delta$ leaves) on vertices $\{x_i\}\cup N_{T_1}(x_i)$ to an arbitrary subset of $S_i$, with $x_i$ mapped to the center $\tilde{x}_i$. This gives us an embedding of the forest of stars $T[\{x_i\}\cup N_{T_1}(x_i)\cup \ldots\cup\{x_s\}\cup N_{T_1}(x_s)]$.

Next we extend our forest by connecting these stars according to
the order $x_1,\dots, x_s$, and obtain an embedding of a subtree of $T_1$
which is the union of the stars and $\langle x_1,\dots, x_{s}\rangle_{T_1}$.
Suppose we have connected the first $i-1$ stars, i.e., we have an embedding of $\langle x_1,\dots, x_{i-1}\rangle_{T_1}$, and now we will connect it to $\tilde{x}_i$, the image of $x_i$.
Recall that $\dist_{T_1}(x_i, \langle x_1,\dots, x_{i-1}\rangle_{T_1}) = 5$ and thus let the path to be embedded be $x_i, y_1, y_2, y_3, y_4, z$.
Note that $x_i, z, y_1$ are already embedded in $H=G\cup G_\alpha$.
Moreover, if $z\in \{x_1,\dots,x_{i-1}\}$, then $y_4$ has already been embedded; otherwise, fix a neighbor of $g(z)$ in $G_\alpha$ which is not covered by the current partial forest as $g(y_4)$.
This is possible because $\delta(G_\alpha) \ge \alpha n$ and $|T_1|\le \alpha n/(2\Delta)$.
Note that, using $G_\alpha$, there are at least $\alpha n/2$ choices for the image of $y_2$ and at least $\alpha n/2$ choices for the image of $y_3$, so, as $G\in \cG(n, D/n, \eps, C)$, we can pick $\tilde y_2$ and $\tilde y_3$ so that $\tilde{y}_2\tilde{y}_3$ is an edge of $G$. Thus, the sequence $\tilde{x}_i,g(y_1), \tilde y_2, \tilde y_3, g(y_4),g(z)$
forms a path in $H$. Define $g(y_i)=\tilde y_i$ for $i=2,3$.
When finished, this completes the second step of the embedding.

For the last step, note that since the partial tree that has been embedded is connected, we can finish the embedding of $T_1$ by iteratively attaching leaves to the partial embedding. This is always possible because $\delta(G_\alpha) \ge \alpha n$ and $|T_1|\le \alpha n/(2\Delta)$.
Let $g$ be the resulting embedding function and $\tilde{T_1}=g(T_1)$.

By the claim for any $u, v\in V$, there are at least $2(\Delta+3)\eps n$
stars from $S_1,\dots, S_s$ such that their centers are in
$N_{G_\alpha}(u)$ and their leaves are in $N_{G_{\alpha}}(v)$.
Since these stars are subtrees of $\tilde{T_1}$, we conclude that $|B_{\tilde{T_1},H}(u,v)|\ge 2(\Delta+3)\eps n$ for any $u, v\in V$, as required.
\end{proof}

\subsection{Almost spanning tree embeddings}
\label{sec:almost}
To extend $T_1$ to the almost spanning tree $T'$, we will use the
following corollary of a tree embedding result of
Haxell~\cite{haxell2001tree} (this is her Theorem~1 with $\ell=1$ and
each $d_i=\Delta$).  We note that it was first observed by Balogh,
Csaba, Pei, and Samotij~\cite{BCPS10} that this is applicable in
sparse random graphs.  For a graph $G$ and vertex set
$X\subseteq V(G)$, we let $N_G(X):=\bigcup_{x\in X}N_G(x)$.

\begin{cor}\label{cor}
Let $T$ be a tree with $t$ edges and maximum degree at most $\Delta$.
Suppose $k\ge 1$ is an integer and $G$ is a graph satisfying the following two conditions:
\begin{enumerate}[label=\rmlabel]
\item $|N_G(X)|\ge \Delta|X|+1$ for every $X\subseteq V(G)$ with $1\le |X|\le 2k$,
\item $|N_G(X)|\ge \Delta |X| + t+1$ for every $X\subseteq V(G)$ with $k<|X|\le 2k+1$.
\end{enumerate}
Then~$G$ contains $T$ as a subgraph.
Moreover, for any vertex $x_0$ of $T$ and any $y\in V(G)$, there exists an embedding $f$ of $T$ into $G$ such that $f(x_0) = y$.
\end{cor}

\section{Main technical result}\label{sec:tech_thm}
In this section we prove our main technical result, Theorem~\ref{thm1}.
Given $T\in \cT(n,\Delta)$ we will use Lemma~\ref{lem1} to obtain a
subtree~$T_1$ of~$T$ of small linear size, which we embed with the help of
Lemma~\ref{lem3} and then extend to the embedding of an almost spanning
subtree of~$T$ using Corollary~\ref{cor}. We then use the reservoir sets
$B_{T,H}(u,v)$ to extend the embedding to cover the last few vertices.

Although we risk being somewhat repetitive, 
with the relevant definitions at
hand, we are able to say more precisely how the sets $B_{T,H}(u,v)$
will help us to embed these last few vertices. 
Suppose we have a partial embedding $g: T'\rightarrow H$ of our tree $T$
into the host graph $H$, such that $T'\subseteq T$ is connected and let
$\tilde{T'}=g(T')$. Since $T'$ is a subtree in $T$ we can extend it vertex
by vertex by connecting $T'$ with some new vertex $b \in V(T \setminus T')$, which has one neighbour in $V(T')$. Assume that this neighbour $a$ of $b$ in $T$ has been embedded to ${u}$, but none of the unused vertices is connected to ${u}$ in $H$ so that we cannot simply embed $b$ to one of the unused vertices. Instead, if there exists an unused vertex $v$ such that $B_{\tilde{T'},H}(u,v)\neq \emptyset$,
then we can proceed with the embedding as follows.
Let $w\in B_{\tilde{T'},H}(u,v)$ and note that, by the definition of
$B_{\tilde{T'},H}(u,v)$, we have $w\in V(\tilde{T'})$.
Let $c=g^{-1}(w)$, and let $g'(x)=g(x)$, for any $x\in V(T')\setminus \{c\}$, $g'(c)=v$ and $g'(b)=w$. Using the definition of $B_{\tilde{T'},H}(u,v)$, this gives a partial embedding $g'$ into $H$ with one more leaf, $b$, embedded.
We will show that we only need this procedure to embed the last $2\eps n$
vertices of $T$, and, for any $u, v\in V$, by the property guaranteed by
Lemma~\ref{lem3}, the reservoir sets $B_{\tilde{T'},H}(u,v)$ will be large enough to proceed greedily.

\begin{proof}
Given $\alpha$, $C$ and $\Delta$, set $\eps'=\alpha^{\Delta+2}C^{-2\Delta}2^{-\Delta-8}\Delta^{-7}$, a constant small enough that by taking $D_0$ and $n_0$ to be large we can use the conclusion of Lemma~\ref{lem3} with $\eps=\eps'$ (cf.\ \eqref{eq:eps_Dz}).
Set $\eps:=\min\{\alpha/(3\Delta),\eps'/(2\Delta)\}$. Suppose then that $D\geq D_0$ and $n\geq n_0$,  $G\in \cG(n, D/n, \eps, C)$ and that $G_\alpha$ is an $n$-vertex graph on $V(G)$ with $\delta(G_\alpha) \ge \alpha n$, and let $T\in \cT(n,\Delta)$.

By Lemma~\ref{lem1} with  $\beta = \alpha/(2\Delta)^2$, there exist subtrees $T_1\subseteq T' \subseteq T$ so that $\alpha n/(2\Delta)^2\leq |T_1|\leq \alpha n/(4\Delta)$, $e(T_1,T\setminus T_1)=1$ and $|T\setminus T'|=2\eps'n$.
We apply Lemma~\ref{lem3} and obtain an embedding $g$ of $T_1$ in $H:=G_\alpha \cup G$ such that $|B_{\tilde{T_1},H}(u,v)|\ge 2(\Delta+3)\eps' n$ for any $u, v\in V$, where $\tilde T_1=g(T_1)$.
Let $a b\in E(T)$ be the unique edge between $T_1$ and $T\setminus T_1$ such that $a\in V(T_1)$, and let $\tilde a=g(a)$.
Define $T'':=T'\setminus (T_1\setminus \{a\})$ and $H':= H\setminus (V(\tilde T_1)\setminus \{\tilde a\})$.

We want to apply Corollary~\ref{cor} to find an embedding $g'$ of $T''$ in
$H'$, with $g'(a) = \tilde a$. So we need to verify the assumptions of Corollary~\ref{cor} with $k=\eps n-1$.
Firstly, note that by $\delta(G_\alpha) \ge \alpha n$ and $|T_1|\le \alpha n/(2\Delta)$, we know that $\delta(H') \ge \alpha n - |T_1| \ge \alpha n/2 \ge \Delta\cdot 2k +1$.
Thus, condition~(i) of Corollary~\ref{cor}  holds for sets on at most $2k$ vertices.
Secondly, we claim that for any set $X\subseteq V(H')$ of size at least
$k+1 = \eps n$ we have $|V(H')\setminus N_{H'}(X)|< \eps n$.
Indeed, since  $G\in \cG(n, D/n, \eps, C)$ and both $X$ and $V(H')\setminus N_{H'}(X)$ are subsets of $V(H)$, if $|V(H')\setminus N_{H'}(X)|\ge \eps n$ then there is an edge in $H$, and hence $H'$, between $X$ and $V(H')\setminus N_{H'}(X)$, a contradiction.
Thus, since $|T_1|-1=|T'| - |T''| = |H| - |H'|$ and $|H| - |T'| = 2\eps' n$, we have $|H'| - |T''| = |H| - |T'| = 2\eps' n$, and thus, as $\eps'\geq 2\Delta \eps$,
\begin{equation*}
|N(X)| \ge |H'| - \eps n = |T''|+(2\eps' -\eps)n> |T''|+\Delta\cdot (2k+1).
\end{equation*}
Thus, we can apply Corollary~\ref{cor} and obtain the embedding $g'$ of $T''$ into $H'$.
Combine $g$ and $g'$ to obtain an embedding $g_0$ of $T'$ in $H$, and write $\tilde T' = g_0(T')$.

For any $u, v, w\in V$ and any two trees $S$ and $S'$, observe that if $N_{S}(w)= N_{S'}(w)$ and $w\in B_{S,H}(u,v)$, then $w\in B_{S',H}(u,v)$.
Since, by construction, for any vertex $w\in V(\tilde T_1)\setminus
\{\tilde a\}$ we have $N_{\tilde T_1}(w)= N_{\tilde T'}(w)$, and so $|B_{\tilde T',H}(u,v)|\ge |B_{\tilde T_1,H}(u,v)|-1\ge 2(\Delta+3)\eps' n-1$ for any $u, v\in V$.

It remains to embed the $2\eps' n$ vertices in $V(T\setminus T')$ to $H$.
We achieve this using $B_{\tilde{T'},H}(u,v)$ as explained at the beginning of this section.
More precisely, since $T'$ is connected, we can obtain~$T$ from $T'$ by iteratively attaching one new leaf at a time, say using the sequence $T':=T_0'\subset T_1'\subset \cdots \subset T_{2\eps n}' = T$.
We claim that we can extend the embedding inductively while keeping $|B_{\tilde T_{i}',H}(u,v)|\ge |B_{\tilde T_{i-1}',H}(u,v)| - (\Delta+3)$ for every $i\in [2\eps' n]$, where each $\tilde T_i'$ is the image of $T_i'$ in $H$.
Indeed, fix some index $i\in [2\eps' n]$ and now we need to attach the vertex $b_i\in V(T_i'\setminus T_{i-1}')$, whose parent $a_i\in T'_{i-1}$ has been embedded to~$\tilde a_i$.
Pick any vertex $v'$ in $V(H)\setminus V(\tilde T_{i-1}')$.
Since
\begin{equation*}
|B_{\tilde T_{i-1}',H}(\tilde a_i, v')|\ge |B_{\tilde T',H}(\tilde a_i, v')| - (i-1)(\Delta+3) >2(\Delta+3)\eps' n - 1 - (i-1)(\Delta+3) >0,
\end{equation*}
we can pick $w\in B_{\tilde T_{i-1}',H}(\tilde a_i, v')$ and let
$c=g_{i-1}^{-1}(w)$. We now swap~$c$ out of the current embedding and use
its previous image~$w$ to embed~$b_i$, and embed $c$ to $v'$ instead. Precisely, define the new embedding $g_i$ by $g_i(x)=g_{i-1}(x)$ for any $x\in V(T_{i-1}')\setminus \{c\}$, $g_i(c)=v'$ and $g_i(b_i)=w$.
Let $\tilde T_i' = g_i(T_i')$.
Note that $N_{\tilde T_i'}(x) = N_{\tilde T_{i-1}'}(x)$ for all but at most  $\Delta+3$ vertices $x$ in $V(\tilde T_{i-1})$: the vertices $\tilde a_i, v', w$ and the neighbors of $w$ in $\tilde T_{i-1}$ -- because they are the vertices that are incident to the edges in $E(\tilde T_i')\setminus E(\tilde T_{i-1})$.
Thus, we have $|B_{\tilde T_{i}',H}(u,v)|\ge |B_{\tilde T_{i-1}',H}(u,v)| - (\Delta+3)$, for any $u,v\in V$, and we are done.
\end{proof}

\section{Tree universality in randomly perturbed dense graphs}
\label{sec:expanderToRandom}
In this section, we show how Theorem~\ref{thm1} implies Theorem~\ref{thm:main_theorem}, using the following simple proposition.
\begin{prop}\label{prop:expander}
For any $\eps>0$ and $C\ge 2$ there exists $D_0$ such that the following holds for any $D\ge D_0$. The random graph $G(n,D/n)$ a.a.s.\ contains some graph $G\in \cG(n,D/n,\eps,C)$.
\end{prop}
\begin{proof}
Choose $D_0$ such that $D_0\ge 10^4\eps^{-2}$.
Let $D\geq D_0$ and $H:=G(n,D/n)$. Note that, by a simple Chernoff
bound, the probability that, for all $U,W\subset V(H)$, with
$|U|,|W|\ge \eps n/10$, we have 
\begin{equation}\label{edgecount}
 3D|U||W|/(4n)\leq e_{H}(U,W)\le 5D|U||W|/(4n)
\end{equation}
is at least $1-2^{2n}e^{-D\eps^2n/4800}=1-o(1)$.
Assume then that the property in \eqref{edgecount} holds. We will show that there are few vertices with high degree in $H$.

Let $A\subset V(H)$ be the set of vertices with degree exceeding $5D/4$ in $H$, and note that it satisfies $e_H(A,V(H))> 5D|A|/4$. Thus, by the property in \eqref{edgecount}, we have that $|A|< \eps n/10$.

 If we delete all the edges incident to vertices of degree larger than $CD\geq 5D/4$ from $H$ then we are left with a graph $G$ of maximum degree at most $CD$ satisfying that for any two sets $U$ and $W$ of size at least $\eps n$, we have
 \[
 e_{G}(U, W)\ge \tfrac{3}{4}\tfrac{D}{n}\cdot |U\setminus A|\cdot |W\setminus A|\ge \tfrac{3}{4}\tfrac{D}{n}\cdot (9|U|/10)\cdot(9|W|/10)\geq \tfrac{1}{C}\tfrac{D}{n}|U||W|.
 \]
Thus, $G$ is in $\cG(n,D/n,\eps,C)$, as required.
\end{proof}

\begin{proof}[Proof of Theorem~\ref{thm:main_theorem}]
Given $\alpha$ and $\Delta$, let $\eps$, $D_0$ and $n_0$ be given by Theorem~\ref{thm1} on inputting $\alpha$, $\Delta$ and $C=2$. We choose $D'_0\ge D_0$ so that Proposition~\ref{prop:expander} with $\eps$ and $C$ is applicable for $D\ge D'_0$. Since a.a.s.\ the random graph $G(n,D/n)$ contains a graph from $\cG(n,D/n,\eps,C)$ we have, by Theorem~\ref{thm1}, that $G_\alpha\cup G(n,D/n)$ is  a.a.s.\ $\cT(n,\Delta)$-universal.
\end{proof}

\section{Concluding remarks}
A graph $G$ is called an $(n,d,\lambda)$-graph if $|G|=n$, $G$ is
$d$-regular and the second largest eigenvalue of the adjacency matrix
of $G$ in absolute value is at most $\lambda$. There is extensive
literature on the properties of $(n,d,\lambda)$-graphs, see, e.g., a
survey of Krivelevich and Sudakov~\cite{krivsud2006survey}. It is
known that $(n,d,\lambda)$-graphs $G$ satisfy the so-called expander
mixing lemma, that is, for all vertex subsets $A,\,B\subset V(G)$, we
have
\[
  \left|e(A,B)-\tfrac{d}{n}|A||B|\right|\le \lambda \sqrt{|A||B|}.
\]
Our main technical result, Theorem~\ref{thm1}, easily implies that,
for any $\alpha$ and $\Delta$, there is some sufficiently small $\eps$
such that, for any sufficiently large~$d$ and $\lambda\leq \eps d/2$,
any union of $G_\alpha$, a graph on $n$ vertices with minimum degree
at least $\alpha n$, with an $(n,d,\lambda)$-graph is
$\cT(n,\Delta)$-universal.



\endgroup
\end{document}